\documentclass[11pt,reqno]{amsart}
\usepackage{amssymb}
\usepackage{palatino}
\usepackage{tikz-cd}
\input amssym.def

\usepackage{xcolor}
\usepackage{amsmath}
\usepackage{amssymb, amsfonts}
\usepackage{amscd}
\usepackage[mathscr]{eucal}
\usepackage{palatino}
\usepackage{theoremref}
\usepackage{tikz}
\usepackage{float}
\usepackage{eqnarray}
\usepackage{comment}
\usepackage[utf8]{inputenc}
\usepackage[T1]{fontenc}

\newfont{\cyrr}{wncyr10}

\newcommand{\thmref}[1]{Theorem~\ref{#1}}

\newcommand{\lemref}[1]{Lemma~\ref{#1}}

\newtheorem{thm}{Theorem}

\newtheorem{lem}[thm]{Lemma}

\newtheorem{rmk}{Remark}[section]

\newcommand{\Z}{{\mathbb Z}}

\newcommand{\G}{{\rm G}}

\def\({\left(}
\def\){\right)}
\def\[{\left[}
\def\]{\right]}

\def\G{{\rm G}}

\def\N{\mathbb{N}}
\def\R{\mathbb{R}}

\def\Q{\mathbb{Q}}

\def\F{\mathbb{F}}

\renewcommand{\mod}{ \text{mod }}

\newcommand\norm[1]{\left\lVert#1\right\rVert}

\title{A Non-abelian Large Sieve and Artin's primitive root conjecture}

\author{M. Ram Murty}
\address{ M. Ram Murty, \newline
Department of Mathematics,
Queen's University, Jeffery Hall, 
99 University Avenue, 
Kingston, ON K7L 3N6, 
Canada}
\email{murty@queensu.ca}

\author{Sunil Naik}
\address{Sunil Naik,
\newline
Max-Planck-Institut f\"{u}r Mathematik
\newline Vivatsgasse 7, 53111, Bonn, Germany
}
\email{naik@mpim-bonn.mpg.de}

\begin{document}
	
\hfuzz 5pt	
	
	
	
\maketitle

%\begin{center}
%    \today
%\end{center}

\begin{abstract}
A well-known conjecture of Artin states that if $a$ is an integer not equal to $0, \pm 1$ or a perfect square, then there exist infinitely many primes $p$ such that $a$ is a primitive root $(\mod p)$. 
In this article, we study a generalization of the classical (abelian) large sieve inequality in non-abelian settings. Assuming the non-abelian large sieve inequality, we provide a proof of Artin's primitive root conjecture. 
Further, using duality techniques, we derive unconditional results towards the conjectured
non-abelian large sieve inequality.
\end{abstract}

\section{Introduction and statements of results}
Gauss in his {\em Disquisitiones Arithmeticae} observed that for any prime $p$ not equal to $2$ and $5$, the period of the decimal expansion of $\frac{1}{p}$ is equal to the 
smallest positive integer $k$ such that
$$
10^k ~\equiv~ 1 ~(\mod p).
$$
This is called the order of $10 ~(\mod p)$. Thus, by Lagrange's theorem, 
the period is maximal if and only if $\F_p^{^*} = \langle 10 \rangle$, in which case, $10$ is called a primitive root $(\mod p)$. 
Gauss raised the question of how often $10$ is a primitive root $(\mod p)$ but made no specific conjecture. The nineteenth-century mathematicians knew that if $p$ is prime of the form $p = 4q+1$ with $q$ prime, 
then $2$ is a primitive root $(\mod p)$. It was not until $1927$ that
a precise conjecture was formulated. In a private conversation 
with H. Hasse in September $1927$, E. Artin \cite{Ar} formulated a precise conjecture 
which states that if $a$ is an integer not equal to $0, \pm 1$ or a perfect square, 
then there exist infinitely many primes $p$ such that $a$ is a primitive root $(\mod p)$. 
Further, there exists a constant $A(a)$ depending on $a$ such that
$$
N_a(x) ~=~ \#\{p \leq x ~:~ a \text{ is a primitive root } (\mod p)\} 
~\sim~ A(a) \frac{x}{\log x}, \phantom{m} \text{as } x \to \infty.
$$
This conjecture is still open, though significant progress has been made in this direction. 
In 1967, assuming the Generalized Riemann Hypothesis (GRH), C. Hooley \cite{Ho} 
proved both Artin's primitive root conjecture and an asymptotic formula for $N_a(x)$ and gave an explicit formula for $A(a)$. 
In 1984, R. Gupta and the first author \cite{GM} proved unconditionally that 
there exists a set of $13$ numbers such that Artin's primitive root conjecture 
is true for at least one of these $13$ numbers. 
The size of this set was reduced to $7$ by R. Gupta, the first author and 
V. Kumar Murty \cite{GMM}. In 1986,  D. R. Heath-Brown \cite{HB} 
refined further the size of this set to $3$ (also see \cite{MR}). 
At the moment, this seems to be the best result using unconditional methods.  

The first author in his Ph.D. thesis \cite{MRP}
introduced a non-abelian large sieve inequality and initiated 
a new approach towards resolving the Artin's primitive root conjecture.

The classical large sieve inequality asserts that
 $$
\sum_{d \leq Q}~~~  \sideset{}{^*}\sum_{\chi (\mod d)} 
\left|
\sum_{n \leq x} a_n \chi(n)
\right|^2
~\leq~
(Q^2 + x) \sum_{n \leq x} |a_n|^2,
$$
where $\sideset{}{^*}\sum_{\chi (\mod d)} $ means that 
the sum is over primitive Dirichlet characters $\chi~(\mod d)$. 
It can be viewed as {\em `an abelian large sieve inequality'} 
since it is essentially an inequality relating to Dirichlet characters 
(which are all one-dimensional). An adelic interpretation involving the Pr\"ufer group can be found in the first author's paper \cite{murty}.

For the sake of simplicity and to fix ideas, we will focus on $a=2$ in this paper.  It will be clear that the method can be generalized for suitable $a$ with obvious modifications.

For any odd prime $q$, let $\G_q$ denote the Galois group of 
$L_q = \Q(e^{\frac{2\pi i}{q}}, 2^{\frac{1}{q}})$ over $\Q$ and 
$\widehat{\G}_q$ denote the set of irreducible characters of $\G_q$. 
It is known that $\widehat{\G}_q$ contains $q$ irreducible characters 
and $(q-1)$ of them are one-dimensional. 
The remaining irreducible character of $\G_q$ denoted by $\varphi_q$ 
is $(q-1)$-dimensional (see subsection \ref{secGalChar}). For $p$ prime unequal to $q$, let
$\sigma_p$ denotes a Frobenius element of $p$ in $\G_q$.
In an analogy with the classical large sieve inequality, 
we will consider the following sum
\begin{equation}\label{eqalgLS}
\sum_{q \leq Q} \sum_{\varphi \in \widehat{\G}_q \atop \varphi \neq I} 
\left| \sideset{}{'}\sum_{p \leq x} \varphi(\sigma_p) \right|^2,
\end{equation}
where $\sideset{}{'}\sum_{p \leq x}$ means that the sum is over primes $p \leq x$ 
satisfying $i_p \leq \sqrt{x} \log x$ and $i_p$ denotes the index of $2~(\mod p)$. 
Note that $\#\{ 2 < p \leq x ~:~ i_p > \sqrt{x} \log x\} ~=~ o\(\pi(x)\) $(see \cite{MR}). 
Here $\sum_{q \leq Q}$ indicates that $q$ varies over odd primes $\leq Q$.

We see that the contribution from the one-dimensional characters of $\G_q$ in \eqref{eqalgLS} is
$$
\sum_{q \leq Q} \sum_{\chi (\mod q) \atop \chi \neq I}
\left| \sideset{}{'}\sum_{p \leq x} \chi(p) \right|^2
~\ll~
\(Q^2 + x \) \pi(x).
$$
We can express $\varphi_q(\sigma_p)$ as an exponential sum (see subsection \ref{secPhiqexp}):
\begin{eqnarray}\label{eqphiqexp}
\varphi_q\(\sigma_p\)
~=~
\frac{1}{q} \sum_{ c\in \F_q^* \atop d \in \F_q} e \( \frac{(p-1)d - i_pc}{q}\),
\end{eqnarray}
where $e(t) = e^{2\pi i t}$. Using \eqref{eqphiqexp}, 
we show that (see Lemma \ref{lemlsieveexp})
if $Q \leq x^{1-\epsilon}$ for some $\epsilon > 0$, then
$$
\sum_{q \leq Q} \sum_{\varphi \in \widehat{\G}_q \atop \varphi \neq I} 
\left| \sideset{}{'}\sum_{p \leq x} \varphi(\sigma_p) \right|^2 
~\ll~ 
(Q^2 + x) \pi(x) ~+~
\sum_{q \leq Q} \sum_{c, d \in \F_q^{{^*}} } 
\left|\sideset{}{'}\sum_{p \leq x}
e\(\frac{pd - i_p c}{q}\) \right|^2.
$$
Heuristically, we expect that
\begin{eqnarray}
\sum_{q \leq Q} \sum_{c, d \in \F_q^{^*} } 
\left|\sideset{}{'}\sum_{p \leq x}
e\(\frac{pd - i_p c}{q}\) \right|^2
~\ll~ (Q^3 + x) \pi(x).
\end{eqnarray}
The above inequality is related to non-abelian characters of 
Galois extensions $L_q$ over $\Q$, hence it may be regarded as 
`{\em a non-abelian large sieve inequality}'. The first author in his Ph.D. thesis (see \cite{MRP}) proved that the non-abelian large sieve inequality 
together with the hypothesis that the Dedekind zeta functions of $L_q$ have no zeros in the region $\Re(s)>3/4$
imply Artin's primitive root conjecture and an asymptotic formula for $N_a(x)$. In this article, we use certain sieve techniques, a seminal idea used by R. Gupta and the first author in \cite{GM}, instead of a quasi-GRH.

In this article, assuming the non-abelian large sieve inequality and using sieve techniques, we prove the following theorem.
\begin{thm}\label{thmNLS-AC}
Suppose that
$$
\sum_{q \leq Q} \sum_{c, d \in \F_q^{^*} } 
\left|\sideset{}{'}\sum_{p \leq x}
e\(\frac{pd - i_p c}{q}\) \right|^2
~\ll~ (Q^3 + x) \pi(x).
$$
Then there exist infinitely many primes $p$ such that 
$2$ is a primitive root $(\mod p)$. Further, we have
$$
\#\{ p \leq x ~:~ 2 \text{ is a primitive root } (\mod p)\} 
~\gg~ \frac{x}{\log^2 x}.
$$
\end{thm}

\begin{rmk}
Let $a \neq 0, \pm 1$ be an integer that is not a perfect square. 
As stated earlier, 
for the sake of simplicity, we have chosen $a = 2$ in \thmref{thmNLS-AC} and similar
arguments will work for any such integer $a$.
\end{rmk}

\begin{rmk}
For any odd prime $q$, let us define
$$
\Lambda_q(n) ~=~
\begin{cases}
\varphi_q(\sigma_p) \log p 
& \text{if } n = p^m, \text{ a prime power}, \\
0 & \text{otherwise},
\end{cases}
$$
and set
$$
\psi(x, \varphi_q) ~=~ \sum_{n \leq x} \Lambda_q(n).
$$
Suppose that the Generalized Riemann Hypothesis (GRH) is true, 
then we can show that 
$$
\psi(x, \varphi_q) ~\ll~ q x^{\frac{1}{2}} \log^2 x
$$
uniformly for $q \leq x^{\frac{1}{4}}$.	Thus, we deduce that
\begin{eqnarray}\label{eqpiqGRH}
\pi(x, \varphi_q)  ~=~ \sum_{p \leq x} \varphi_q(\sigma_p) 
~\ll~ q x^{\frac{1}{2}} \log x
\end{eqnarray}
uniformly for $q \leq x^{\frac{1}{4}}$ (see \cite[Theorem 1.1]{LO}).
Hence if $Q \leq x^{\frac{1}{4}}$, then under the assumption of GRH, 
we deduce that
\begin{eqnarray}\label{eqnlsGRH}
\sum_{q \leq Q} \left| \sum_{p \leq x} \varphi_q(\sigma_p) \right|^2 
~\ll~
Q^3 x \log^2 x.
\end{eqnarray}
We remark that if the inequality in \eqref{eqnlsGRH} is true, 
then arguing as in the proof of \thmref{thmNLS-AC}, we can deduce that 
there exist infinitely many primes $p$ such that $2$ is a primitive root $(\mod p)$. 
Further, we can show that the number of primes $p$ for which 
$2$ is a primitive root $(\mod p)$ is $\gg \frac{x}{\log^2 x}$.
This shows that only the GRH ``on average'' is needed in resolving Artin's conjecture.
\end{rmk}

We also study non-abelian large sieve inequalities with arbitrary coefficients. 
In this context, unconditionally, we prove the following theorem.
\begin{thm}\label{thmN-abLS} 
Let $f$ be an integer-valued arithmetic function satisfying 
$1 \leq f(n) \leq C x^{\theta}$
for some positive real numbers $\theta$, $C$ and for all $n \leq x$.
Let $\(a_n\)_{n \geq 1}$ be a sequence of complex numbers. Then we have
$$
\sum_{q \leq Q} \sum_{c , d\in \F_q^{^*}} 
\left| \sum_{n \leq x} a_n ~e\(\frac{nd + f(n)c}{q}\) \right|^2
~\ll~
(Q+x^\theta)(Q^2+x) \sum_{n \leq x} |a_n|^2.
$$
\end{thm}

As a consequence of \thmref{thmN-abLS}, we have the following theorem.
\begin{thm}\label{thmNLSunc}
We have
\begin{eqnarray*}
\sum_{q \leq Q} \sum_{c, d \in \F_q^{^*} } 
\left|\sideset{}{'}\sum_{p \leq x}
e\(\frac{pd - i_p c}{q}\) \right|^2
~\ll~ (Q + x^{\frac{1}{2}} \log x) (Q^2 + x) \pi(x).
\end{eqnarray*}
\end{thm}

\medspace

Note that the index function $i_p$ satisfies $i_p \mid (p-1)$, 
thus we impose the following restrictions on $f$:
\begin{itemize}
\item[(i)] There exists a positive constant $C> 0$ and $\theta \in (0, 1)$ 
such that $ 1 \leq f(n) \leq C x^\theta$ for any $n \leq x$;
\item[(ii)] $f(n) \mid (n-1)$.
\end{itemize}
In this set-up, we prove the following theorem.
\begin{thm}\label{thmNLS}
Suppose that $Q \gg \log x$. Then we have
$$
\sum_{q \leq Q} \sum_{c, d \in \F_q^{^*}} \left|
\sum_{n \leq x \atop (q, f(n))=1} a_n ~e\( \frac{nd+f(n)c}{q} \)
\right|^2
~\ll~(Qx + Q^3 x^\theta \log Q) \sum_{n \leq x} |a_n|^2
$$
for any sequence $(a_n)_n$ of complex numbers.
\end{thm}

We observe that $\varphi_q(\sigma_p) \neq 0$ implies that 
$p \equiv 1 (\mod q)$. Further, if $p \equiv 1 (\mod q)$, 
then from \eqref{eqphisigmap}, we have
\begin{eqnarray*}
\varphi_q(\sigma_p) ~=~ \sum_{c \in \F_{q}^{*}} e\(\frac{i_p c}{q}\).
\end{eqnarray*}
Using this we show that (see Lemma \ref{EqPhiq=1(q)})
$$
\sum_{q \leq Q} \left| \sum_{p \leq x} \varphi_q(\sigma_p) \right|^2
~\leq~ 
Q \sum_{q \leq Q} \sum_{c \in \F_{q}^{^*}} \left|
\sum_{p \leq x \atop p \equiv 1 (\mod q)} e\(\frac{i_p c}{q}\)
\right|^2.
$$
In this context, we have the following result.
\begin{thm}\label{thm-NLS-opt}
Suppose that $Q \leq x^{1 - \epsilon}$ for some $\epsilon > 0$. Then we have
\begin{eqnarray*}
\sum_{q \leq Q} \sum_{c \in \F_{q}^{^*}} \left|
\sum_{p \leq x \atop p \equiv 1 (\mod q)} a_p~e\(\frac{i_p c}{q}\)
\right|^2
~\ll_{\epsilon}~ \frac{x}{\log\log x} \sum_{p \leq x} |a_p|^2
\end{eqnarray*}
for any sequence of complex numbers $(a_p)_p$. 
Suppose that GRH is true and $\log Q = o(\log x)$, 
then there exists a sequence $(b_p)_p$ of complex numbers such that
\begin{eqnarray*}
\sum_{q \leq Q} \sum_{c \in \F_{q}^{^*}} \left|
\sum_{p \leq x \atop p \equiv 1 (\mod q)} b_p~e\(\frac{i_p c}{q}\)
\right|^2
~\gg~ x \cdot \frac{\log\log Q}{\log x} \cdot  \sum_{p \leq x} |b_p|^2.
\end{eqnarray*}
\end{thm}

\medspace

\section{Prerequisites}

\subsection{The Galois group of $\Q(e^{\frac{2\pi i}{q}}, 2^{\frac{1}{q}})$ over $\Q$ and its characters}\label{secGalChar}
For any odd prime $q$, let $L_q = \Q(\zeta_q, 2^{\frac{1}{q}})$ 
and ${\rm G}_q = {\rm Gal}(L_q/\Q)$, where $\zeta_q = e^{\frac{2\pi i}{q}}$. 
Any element $\sigma \in {\rm G}_q$ is completely determined by 
its values on  $\zeta_q$ and  $2^{\frac{1}{q}}$:
\begin{eqnarray*}
\sigma(\zeta_q) ~=~ \zeta_{q}^{a}, ~~ \sigma(2^{\frac{1}{q}})
~=~ \zeta_{q}^{b} 2^{\frac{1}{q}},\phantom{mm} a \in \F_q^{^*},~~ b \in \F_q.
\end{eqnarray*}
It is easy to see that $\G_q$ is isomorphic to the following subgroup of ${\rm GL}_2(\F_q)$:
\begin{eqnarray*}
\left\{
\begin{bmatrix}
	a & b \\
	0 & 1
\end{bmatrix}
~:~
a \in \F_q^{^*},~~ b \in \F_q
\right\}.
\end{eqnarray*}
For any $(a, b), (r, s) \in \F_q^{^*} \times \F_q$, we note that
\begin{eqnarray*}
\begin{bmatrix}
r & s \\ 0 & 1\\
\end{bmatrix}
\begin{bmatrix}
a & b \\ 0 & 1\\
\end{bmatrix}
\begin{bmatrix}
r & s \\ 0 & 1\\
\end{bmatrix}^{-1}
~=~
\begin{bmatrix}
a & br+ (1-a)s \\ 0 & 1\\
\end{bmatrix}.
\end{eqnarray*}
Hence, we can see that the group $\G_q$ has exactly $q$ conjugacy classes which can be 
described as follows (see \cite{MRP} for more details):
\begin{eqnarray*}
I &~=~& \left\{ 
\begin{bmatrix}
    1 & 0 \\
	0 & 1
\end{bmatrix}
\right\},\\
C_a &~=~& \left\{ 
\begin{bmatrix}
	a & b \\
	0 & 1
\end{bmatrix}
~:~ b \in \F_q 
\right\},\phantom{mm}
a \in \{2, 3, \cdots, q-1\}, \\
D &~=~&
\left\{ 
\begin{bmatrix}
	1 & b \\
	0 & 1
\end{bmatrix}
~:~ b \in \F_{q}^{^*}
\right\}.
\end{eqnarray*}
Thus $\G_q$ has $q$ irreducible characters and $(q-1)$ of them are one-dimensional 
which are related to the Dirichlet characters $(\mod q)$ (see \cite{Go, MRP}) :
Let $\chi_i ~ (i =1, 2, \cdots, q-1)$ be Dirichlet characters $(\mod q)$. 
Then define
\begin{eqnarray}\label{eqphii}
\varphi_i\left( \begin{bmatrix}
	a & b \\
	0 & 1
\end{bmatrix} \right) 
~=~ \chi_i(a),
\phantom{mm} a \in \F_q^{^*},~~ b \in \F_q.
\end{eqnarray}
 Let us denote the remaining irreducible character of $\G_q$ by $\varphi_q$. 
 Then $\varphi_q$ is  $(q-1)$-dimensional (see \cite[Corollary 2, p. 18]{Se}) 
 and using orthogonality relations, we see that
 \begin{eqnarray*}
 \varphi_q (I) ~=~ q-1, ~~ \varphi_q(C_a) ~=~ 0,~~ \varphi_q(D) ~=~ -1.
 \end{eqnarray*}
 In fact, we can express the values of $\varphi_q$ as an exponential sum:
 \begin{eqnarray}\label{eqphiexpsum}
 \varphi_q\(
\begin{bmatrix}
	a & b \\
	0 & 1
\end{bmatrix} \)
~=~
\frac{1}{q} \sum_{ c\in \F_q^* \atop d \in \F_q} e \( \frac{(a-1)d - bc}{q}\),
 \end{eqnarray}
 where $e(t) = e^{2\pi i t}$.
 
 \medspace

\subsection{The Artin symbols in $\G_q$ and values of $\varphi_i$ at these elements}\label{secPhiqexp}
Let $p$ be an odd prime relatively prime to $q$ and $\sigma_p$ be a 
Frobenius element of $p$ in $\G_q$. Suppose that $p \equiv a ~(\mod q)$ 
with $a \in \{1, 2,\cdots, q-1\}$. Since $\Q(\zeta_q) \subset L_q$, 
we see that $\sigma_p(\zeta_q) = \zeta_{q}^{a}$ and hence
$$
\sigma_p ~=~ \begin{bmatrix}
	a & b \\
	0 & 1
\end{bmatrix}
$$
 for some $b \in \F_q$. If $a \neq 1$, then we conclude that 
 the Artin symbol $\(\frac{L_q/\Q}{p}\)$ of $p$ in $\G_q$ is equal to $C_a$. 
 Now we assume that $p \equiv 1 ~(\mod q)$. It is well known that $\sigma_p = I$ 
 if and only if $p$ splits completely in $L_q$ (see \cite[p. 18]{La}). 
 Hence we deduce that the Artin symbol of $p$ in $\G_q$ is given by
 \begin{eqnarray*}
 \left( \frac{L_q/\Q}{p} \right) ~=~
 \begin{cases}
 C_a & \text{if } p \equiv a \not\equiv 1 ~(\mod q), \\
 I & \text{if $p$ splits completely in $L_q$},\\
 D & \text{otherwise.}
 \end{cases}
 \end{eqnarray*}
Note that for abelian characters of $\G_q$, we have
\begin{eqnarray*}
\varphi_i(\sigma_p) ~=~ \chi_i(p), \phantom{ mm} i = 1, 2, \cdots, q-1.
\end{eqnarray*}
 Let $i_p = [\F_p^* : \langle 2 \rangle]$ denote the index of $2~(\mod p)$. 
 Then we have $q \mid i_p$ if and only if $q \mid p-1$ and 
 $2^{\frac{p-1}{q}} \equiv 1 ~(\mod p)$. Hence from Dedekind's theorem, 
 we have $q \mid i_p$ if and only if $p$ splits completely in $L_q$ (see \cite[p. 62]{MR}). 
 Thus we can take
 $$
 \sigma_p ~=~ 
 \begin{bmatrix}
 p & i_p \\
 0 & 1
 \end{bmatrix}.
 $$
Therefore from \eqref{eqphiexpsum}, we get
\begin{eqnarray}\label{eqphisigmap}
\varphi_q(\sigma_p) ~=~
\frac{1}{q} \sum_{ c\in \F_q^* \atop d \in \F_q} e \( \frac{(p-1) d - i_p c}{q}\).
\end{eqnarray}
 
 \medspace

 \subsection{Prerequisites from the classical large sieve inequality}
 Consider the following trigonometric sum:
 $$
 S(t) ~=~ \sum_{n\leq x} a_n e^{2 \pi i n t},
 $$
 where $\(a_n\)_{n \geq 1}$ is a sequence of complex numbers and $t$ is a real number.
 From \cite[Theorem 4, Section 2]{Bo}, we have the following result.
 \begin{thm}\label{thmLSdelta}
 Let $\delta \in (0, 1)$ be a real number. 
 Suppose the points $\(x_j\)_{1 \leq j \leq R}$ are $\delta$-spaced, 
 then we have
 $$
 \sum_{j=1}^{R} |S(x_j)|^2 ~\leq~ 
 \(x + \frac{1}{\delta}\)  \sum_{n \leq x} |a_n|^2.
 $$
 \end{thm}
 Note that  numbers in the set 
 $\left\{\frac{a}{d} ~:~ 1 \leq a \leq d \right\}$ are $\frac{1}{d}$-spaced. 
 Further, note that the numbers in the set 
 $\left\{\frac{a}{d} ~:~ 1 \leq a \leq d,~ (a, d)=1,~ 1 \leq d \leq Q \right\}$ 
 are $\frac{1}{Q^2}$-spaced. 
 Hence, as a consequence of \thmref{thmLSdelta}, we have the following results.
 \begin{thm}\label{thmLS1}
 Let $d$ be a positive integer. We have
 $$
 \sum_{a=1}^{d} \left| \sum_{n \leq x} a_n e^{\frac{2 \pi i a n}{d}} \right|^2 
 ~\leq~ 
 \(d + x\) \sum_{n \leq x} |a_n|^2.
 $$
 \end{thm}
 
 \begin{thm}\label{thmLS2}
 We have
 $$
 \sum_{d \leq Q} \sum_{a=1 \atop (a, d)=1}^{d} 
 \left| \sum_{n \leq x} a_n e^{\frac{2 \pi i a n}{d}} \right|^2
 ~\leq~
 \(Q^2 + x\) \sum_{n \leq x} |a_n|^2.
 $$
 \end{thm}
 From \cite[Theorem 7, Section 4]{Bo}, we have the following result.
 \begin{thm}\label{thmLS}
 We have
 $$
 \sum_{d \leq Q}~~~  \sideset{}{^*}\sum_{\chi (\mod d)} 
 \left|
 \sum_{n \leq x} a_n \chi(n)
 \right|^2
 ~\leq~
 (Q^2 + x) \sum_{n \leq x} |a_n|^2,
 $$
 where $\sideset{}{^*}\sum_{\chi (\mod d)} $ means that the sum is over 
 primitive Dirichlet characters $\chi~(\mod d)$.
 \end{thm}
 The inequality in \thmref{thmLS} is famously known as `{\em a  large sieve inequality}'. 
 This inequality is related to Dirichlet characters $\chi ~(\mod d)$ which are one-dimensional, 
 hence it  may be regarded as `{\em an abelian large sieve inequality}'. 
 In the forthcoming sections, we will formulate `{\em a non-abelian large sieve inequality}' 
 which is related to non-abelian characters of certain Galois extensions. 
 We will also discuss some of its applications.

 The following duality principle is one of the main tools in proving a large sieve inequality.
 The interested reader can find a proof in \cite{co-mu}.
 
 \begin{thm}\label{thmDual}
 Let $(c_{ij})$ be complex numbers for $1 \leq i \leq n$, $1 \leq j \leq m$ 
 and $\Delta$ be a positive real number. Then we have
 $$
\sum_{1 \leq j \leq m} 
\left| \sum_{1 \leq i \leq n} c_{ij} a_i \right|^2
~\leq~
\Delta \sum_{1 \leq i \leq n} |a_i|^2
$$
for any sequence $(a_i)_{i=1}^{n}$ of complex numbers if and only if
$$
\sum_{1 \leq i \leq n}
\left| \sum_{1 \leq j \leq m}   c_{ij} b_j \right|^2
~\leq~
\Delta \sum_{1 \leq j \leq m} |b_j|^2
$$
for any sequence $(b_j)_{j = 1}^{m}$ of complex numbers.
 \end{thm}
 
 \medspace

\section{Non-abelian large sieve inequality and Artin's primitive root conjecture}

In this section, we will establish a connection between the non-abelian large sieve inequality 
and Artin's primitive root conjecture, and we will show that 
the non-abelian large sieve inequality implies Artin's primitive root conjecture.
As mentioned in the introduction, we consider the following sum
$$
\sum_{q \leq Q} \sum_{\varphi \in \widehat{\G}_q \atop \varphi \neq I} 
\left| \sideset{}{'}\sum_{p \leq x} \varphi(\sigma_p) \right|^2,
$$
where  $q$ varies over odd primes $\leq Q$ and $\sideset{}{'}\sum_{p \leq x}$ means that the sum is over primes 
$p \leq x$ satisfying $i_p \leq \sqrt{x}\log x$.

In this context, we prove the following lemma.
\begin{lem}\label{lemlsieveexp}
Suppose that $Q \leq x^{1-\epsilon}$ for some $\epsilon > 0$. Then we have
$$
\sum_{q \leq Q} \sum_{\varphi \in \widehat{\G}_q \atop \varphi \neq I} 
\left| \sideset{}{'}\sum_{p \leq x} \varphi(\sigma_p) \right|^2 
~\ll~ 
(Q^2 + x) \pi(x) ~+~
\sum_{q \leq Q} \sum_{c, d \in \F_q^{{^*}} } 
\left|\sideset{}{'}\sum_{p \leq x}
e\(\frac{pd - i_p c}{q}\) \right|^2.
$$
\end{lem}
\begin{proof}
By using the explicit description of irreducible characters of $\G_q$ 
(see Section \ref{secGalChar}), we get
\begin{eqnarray*}
\sum_{q \leq Q} \sum_{\varphi \in \widehat{\G}_q \atop \varphi \neq I} 
\left| \sideset{}{'}\sum_{p \leq x} \varphi(\sigma_p) \right|^2
~=~
\sum_{q \leq Q} \sum_{\chi (\mod q) \atop \chi \neq I}
\left| \sideset{}{'}\sum_{p \leq x} \chi(p) \right|^2
~+~
\sum_{q \leq Q} \left| \sideset{}{'}\sum_{p \leq x} \varphi_q(\sigma_p) \right|^2.
\end{eqnarray*}
By the classical large sieve inequality (see \thmref{thmLS}), we get
\begin{eqnarray*}
\sum_{q \leq Q} \sum_{\chi (\mod q) \atop \chi \neq I}
\left| \sideset{}{'}\sum_{p \leq x} \chi(p) \right|^2
~\ll~
\(Q^2 + x \) \pi(x).
\end{eqnarray*}
From \eqref{eqphisigmap}, we get
\begin{eqnarray}
\sum_{q \leq Q} \left| \sideset{}{'}\sum_{p \leq x} \varphi_q(\sigma_p) \right|^2 
&~=~& \sum_{q \leq Q}   \left| \sideset{}{'}\sum_{p \leq x} \frac{1}{q} 
\sum_{(c, d) \in \F_q^{^*} \times \F_q} 
e\(\frac{(p-1) d - i_p c}{q}\) \right|^2  \notag\\ \label{eqphiqd0}
&~\ll~& \sum_{q \leq Q} \frac{1}{q^2} \left| \sideset{}{'}\sum_{p \leq x}
\sum_{c, d \in \F_q^{^*} }
e\(\frac{(p-1) d - i_p c}{q}\) \right|^2
~+~ \sum_{q \leq Q} \frac{1}{q^2} \left| \sideset{}{'}\sum_{p \leq x} 
\sum_{c \in \F_q^{^*} } 
e\(\frac{ - i_p c}{q}\) \right|^2.
\end{eqnarray}
The latter sum in the above equation is
\begin{eqnarray*}
\notag
\sum_{q \leq Q} \frac{1}{q^2} \left| \sideset{}{'}\sum_{p \leq x} \sum_{c \in \F_q^{^*} } 
e\(\frac{ - i_p c}{q}\) \right|^2
&~=~& 
\sum_{q \leq Q} \frac{1}{q^2} \left| \sideset{}{'}\sum_{p \leq x \atop q \mid i_p} (q-1) 
~+~  \sideset{}{'}\sum_{p \leq x \atop q \nmid i_p} (-1) \right|^2 \\
~~&\ll&~
\sum_{q \leq Q} \frac{1}{q^2} 
\( \sum_{p \leq x \atop q \mid i_p} q \)^2 ~+~ \pi(x)^2.
\end{eqnarray*}
Note that $q \mid i_p$ implies that $p \equiv 1 ~(\mod q)$. 
By using the Brun-Titchmarsh inequality \cite[Theorem 3.8, p. 110]{HR}, we get
\begin{equation}\label{eqd0BT}
\sum_{p \leq x \atop q \mid i_p} 1 ~\leq~ \sum_{p \leq x \atop p \equiv 1 (\mod q)} 1 
~\ll~ \frac{\pi(x)}{q}.
\end{equation}
From \eqref{eqd0BT}, we get
\begin{equation}\label{eqexpd0}
\sum_{q \leq Q} \frac{1}{q^2} \left| \sideset{}{'}\sum_{p \leq x} \sum_{c \in \F_q^{^*} } 
e\(\frac{ - i_p c}{q}\) \right|^2
~\ll~ \pi(x)^2.
\end{equation}
By the Cauchy-Schwarz inequality, we deduce that
\begin{eqnarray}\label{eqexpCS}
 \sum_{q \leq Q} \frac{1}{q^2} \left| \sideset{}{'}\sum_{p \leq x}
\sum_{c \in \F_q^{^*} } \sum_{d \in \F_q^{^*}}
e\(\frac{(p-1) d - i_p c}{q}\) \right|^2
~\ll~
\sum_{q \leq Q} \sum_{c, d \in \F_q^{^*} } 
\left|\sideset{}{'}\sum_{p \leq x}
e\(\frac{pd - i_p c}{q}\) \right|^2.
\end{eqnarray}
This completes the proof of \lemref{lemlsieveexp}.
\end{proof}

\medspace

\subsection{Proof of \thmref{thmNLS-AC}}
Let $\delta > 0$ be a sufficiently small real number. 
From \cite{GM} (also see \cite[p. 66]{MR} and \cite[p. 3]{JM}), 
the cardinality of the following set of primes
\begin{eqnarray*}
S(x) ~&=&~ \left\{ p \leq x ~:~ \frac{p-1}{2} \text{ is prime or }  
\frac{p-1}{2} ~=~  q_1 q_2, ~~ q_i \text{ prime }, \right.\\
&~& \hspace{2cm} \left. x^{0.26} < q_1 < x^{\frac{1}{2} - \delta} < q_2  
\text{ and } \(\frac{2}{p}\) ~=~ -1 \right\}
\end{eqnarray*}
is
\begin{eqnarray}\label{eqSxp-1}
\#S(x) ~\gg~ \frac{x}{\log^2 x}.	
\end{eqnarray}
Let $p \in S(x)$. If $\frac{p-1}{2}$ is a prime, then it is easy to see that 
$2$ is a primitive root $(\mod p)$. Suppose that $p-1 = 2 q_1 q_2$ as above. 
Since $\(\frac{2}{p}\) ~=~ -1$, we see that the order of $2~(\mod p)$ is either 
$2q_1$, $2q_2$ or $2q_1q_2$. If the order of $2~(\mod p)$ is $2q_1$, 
then $p \mid (2^{2q_1} -1)$ and the number of such primes is bounded by
\begin{eqnarray}\label{eqomegq1}
\sum_{q_1 < x^{\frac{1}{2} - \delta}} \omega\(2^{2q_1} -1\)
~\ll~ 
\sum_{q_1 < x^{\frac{1}{2} - \delta}}  q_1
~\ll~ 
x^{1-2\delta}.	
\end{eqnarray}
Hence from \eqref{eqSxp-1} and \eqref{eqomegq1}, if
$$
S_1(x) ~=~ \{ p \in S(x) ~:~ \text{order of $2~(\mod p)$ is either $2q_2$ or $p-1$}\},
$$
then we have
\begin{equation}\label{eqordq2p-1}
\#S_1(x)
~\gg~\frac{x}{\log^2 x}.
\end{equation}
We also have $i_p \leq x^{\frac{1}{2}} \log x$ for $p \in S_1(x)$. 
Note that if the order of $2~(\mod p)$ is $2q_2$, 
then $p$ splits completely in $L_{q_1}$. Thus we have
\begin{equation}\label{eqordq2}
\#\{ p \in S_1(x) ~:~ \text{order of $2~(\mod p)$ is $2q_2$}\}
~\leq~
\sum_{x^{0.26} < q < x^{\frac{1}{2}-\delta}} \tilde{\pi}_1(x, L_{q}/\Q),
\end{equation}
where $\tilde{\pi}_1(x, L_{q}/\Q)$ denotes the number of primes $p \leq x$ 
which split completely in $L_q$ and $i_p \leq x^{\frac{1}{2}} \log x$.
We have the following expression for $\tilde{\pi}_1(x, L_{q}/\Q)$:
$$
\tilde{\pi}_1(x, L_q/\Q) ~=~ \frac{1}{|\G_q|} \sum_{\varphi} \varphi(1) \tilde{\pi}(x,\varphi),
$$
where $\varphi$ varies over irreducible characters of $\G_q$ and
$$
\tilde{\pi}(x, \varphi) ~=~ \sideset{}{'}\sum_{p \leq x} \varphi(\sigma_p).
$$
Here $\sigma_p$ denotes a Frobenius element of $p$ in $\G_q$. We have
$$
\tilde{\pi}_1(x, L_q/\Q) 
~=~ \frac{1}{|\G_q|} \sum_{\varphi} \varphi(1) \tilde{\pi}(x,\varphi)
~=~ \frac{1}{q(q-1)}  \sum_{\varphi \neq 1} \varphi(1) \tilde{\pi}(x,\varphi)
~+~ O\(\frac{\pi(x)}{q(q-1)} \).
$$
Note that
$$
\sum_{x^{0.26} < q < x^{\frac{1}{2}-\delta}} \frac{\pi(x)}{q(q-1)} 
~\ll~ x^{0.74}.
$$
By the Cauchy-Schwarz inequality, we have 
\begin{eqnarray*}
\sum_{x^{0.26} < q < x^{\frac{1}{2}-\delta}}
\frac{1}{q(q-1)} \sum_{\varphi \neq 1} \varphi(1) \tilde{\pi}(x,\varphi)
~&=&~ 
\sum_{x^{0.26} < q < x^{\frac{1}{2}-\delta}}
\sum_{\varphi \neq 1} \frac{\varphi(1)}{\sqrt{q(q-1)}} \frac{\tilde{\pi}(x,\varphi)}{\sqrt{q(q-1)}} \\
~&\leq&~ 
\( \sum_{x^{0.26} < q < x^{\frac{1}{2}-\delta}}\sum_{\varphi \neq 1} 
\frac{\varphi(1)^2}{q(q-1)}\)^{\frac{1}{2}} 
\(\sum_{x^{0.26} < q < x^{\frac{1}{2}-\delta}}\sum_{\varphi \neq 1}  
\frac{|\tilde{\pi}(x,\varphi)|^2}{q(q-1)}\)^{\frac{1}{2}}.
\end{eqnarray*}
Since 
$\sum_{\varphi} \varphi(1)^2 ~=~ q (q-1)$,
we get
$$
\sum_{x^{0.26} < q < x^{\frac{1}{2}-\delta}}\sum_{\varphi \neq 1} \frac{\varphi(1)^2}{q(q-1)} 
~\ll~ \sum_{x^{0.26} < q < x^{\frac{1}{2}-\delta}} 1
~\ll~ x^{\frac{1}{2}-\delta}.
$$
We have
$$
\sum_{x^{0.26} < q < x^{\frac{1}{2}-\delta}}\sum_{\varphi \neq 1}  
\frac{|\tilde{\pi}(x,\varphi)|^2}{q(q-1)}
~\ll~ 
\sum_{x^{0.26} < q < x^{\frac{1}{2}-\delta}} ~~
\sideset{}{^*}\sum_{\chi (\mod q)}  \frac{|\tilde{\pi}(x,\chi)|^2}{q^2} 
~+~ \sum_{x^{0.26} < q < x^{\frac{1}{2}-\delta}} 
\frac{|\tilde{\pi}(x, \varphi_q)|^2}{q^2},
$$
where 
$$
\tilde{\pi}(x, \chi) ~=~ \sideset{}{'}\sum_{p \leq x} \chi(p).
$$
By the classical large sieve inequality and partial summation, we deduce that
$$
\sum_{x^{0.26} < q < x^{\frac{1}{2}-\delta}} \frac{1}{q^2}
\sideset{}{^*}\sum_{\chi (\mod q)} 
\left| \sideset{}{'}\sum_{p \leq x} \chi(p)\right|^2
~\ll~ x^{\frac{1}{2}-\delta} \pi(x).
$$
For any $t \leq x^{\frac{1}{2}-\delta}$, let
$$
A(t) ~=~ \sum_{q \leq t} |\tilde{\pi}(x,\varphi_q)|^2.
$$
From \eqref{eqphiqd0}, \eqref{eqexpd0}, \eqref{eqexpCS} 
and our assumption in \thmref{thmNLS-AC}, we get
$$
A(t) ~\ll~ \pi(x)^2 ~+~ \sum_{q \leq t} \sum_{c, d \in \F_q^{^*}} 
\left| \sideset{}{'}\sum_{p \leq x}  e\(\frac{pd -i_p c}{q}\) \right|^2
~\ll~ (t^3+ x) \pi(x).
$$
Set $D = x^{0.26}, Q = x^{\frac{1}{2}-\delta}$. 
By partial summation, we get
\begin{eqnarray*}
\sum_{x^{0.26} < q \leq x^{\frac{1}{2}-\delta}}  
\frac{|\tilde{\pi}(x, \varphi_q)|^2}{q^2}
&~=~&
\frac{A(Q)}{Q^2} ~-~ \frac{A(D)}{D^2} ~+~ \int_{D}^{Q} \frac{2A(t)}{t^3} dt 
~\leq ~
\frac{A(Q)}{Q^2} ~+~ \int_{D}^{Q} \frac{2A(t)}{t^3} dt 
~\ll ~
x^{\frac{1}{2}-\delta} \pi(x).
\end{eqnarray*}
Thus we have
$$
\sum_{x^{0.26} < q < x^{\frac{1}{2}-\delta}}\frac{1}{q(q-1)} 
\sum_{\varphi \neq 1} \varphi(1) \tilde{\pi}(x,\varphi)
~\ll~ 
\(x^{\frac{1}{2}-\delta}\)^{\frac{1}{2}} 
\(x^{\frac{1}{2}-\delta} \pi(x)\)^{\frac{1}{2}}
~\ll~ 
x^{1- \delta}.
$$
Hence we conclude that
\begin{eqnarray}\label{eqpi1}
\sum_{x^{0.26} < q < x^{\frac{1}{2}-\delta}} \tilde{\pi}_1(x, L_{q}/\Q) 
~\ll~ x^{1- \delta}.	
\end{eqnarray}
From \eqref{eqordq2p-1}, \eqref{eqordq2} and \eqref{eqpi1}, we deduce that
$$
\#\{p \leq x ~:~  2 \text{ is a primitive root } (\mod p)\} 
~\gg~ \frac{x}{\log^2 x}.
$$
This completes the proof of \thmref{thmNLS-AC}. \qed

\medspace

\medspace

\section{Non-abelian large sieve inequality}
In this section, we will study non-abelian large sieve inequality 
with arbitrary coefficients and derive some unconditional results.  
\subsection{Proof of \thmref{thmN-abLS}}
We have
\begin{eqnarray}\label{eqsepcd}
\nonumber 
\sum_{q \leq Q} \sum_{c, d \in \F_q^{^*}} 
\left| \sum_{n \leq x} a_n ~e\(\frac{nd + f(n)c}{q}\) \right|^2
~&=&~ 
\sum_{q \leq Q} \sum_{c, d \in \F_q^{^*}} 
\left| \sum_{m \leq C x^\theta} \sum_{n \leq x \atop f(n) = m} 
a_n ~e\(\frac{nd + mc}{q}\) \right|^2 \\
~&=&~ 
\sum_{q \leq Q} \sum_{c , d\in \F_q^{^*}} 
\left| \sum_{m \leq C x^\theta}  b_m ~e\(\frac{mc}{q}\) \right|^2,
\end{eqnarray}
where
$$
b_m ~=~ \sum_{n \leq x \atop f(n) = m} a_n ~e\(\frac{nd}{q}\).
$$
By \thmref{thmLS1}, we have
\begin{equation}
\sum_{c \in \F_q^{^*}}
\left| \sum_{m \leq C x^\theta}  b_m ~e\(\frac{mc}{q}\) \right|^2
~\ll~
(q + x^\theta ) \sum_{m \leq C x^\theta}  |b_m|^2.
\end{equation}
Thus we get
\begin{eqnarray}\label{eqqcdbm}
\nonumber
\sum_{q \leq Q} \sum_{c, d \in \F_q^{^*}} 
\left| \sum_{m \leq C x^\theta}  b_m ~e\(\frac{mc}{q}\) \right|^2
~&\ll&~
\sum_{q \leq Q} \sum_{d \in \F_q^{^*}}
(q + x^\theta ) \sum_{m \leq C x^\theta}  |b_m|^2 \\
~&\ll&~ 
(Q + x^\theta ) \sum_{m \leq C x^\theta}
\sum_{q \leq Q} \sum_{d \in \F_q^{^*}}
\left| \sum_{n \leq x \atop f(n) = m} a_n ~e\(\frac{nd}{q}\) \right|^2.
\end{eqnarray}
By \thmref{thmLS2}, we get
\begin{equation}\label{eqqcLS}
\sum_{q \leq Q} \sum_{d \in \F_q^{^*}}
\left| \sum_{n \leq x \atop f(n) = m} a_n ~e\(\frac{nd}{q}\) \right|^2
~\ll~
(Q^2+x) \sum_{n \leq x \atop f(n) = m} |a_n|^2.
\end{equation}
By \eqref{eqqcdbm} and \eqref{eqqcLS}, we get
\begin{eqnarray*}
\sum_{q \leq Q} \sum_{c , d \in \F_q^{^*}} 
\left| \sum_{m \leq C x^\theta}  b_m ~e\(\frac{mc}{q}\) \right|^2
~&\ll&~
(Q + x^\theta) \sum_{m \leq C x^\theta}
(Q^2+x) \sum_{n \leq N \atop f(n) = m} |a_n|^2 \\
~&\ll&~
(Q + x^\theta) (Q^2+x)
\sum_{m \leq C x^\theta} \sum_{n \leq x \atop f(n) = m} |a_n|^2 \\
~&\ll&~
(Q + x^\theta) (Q^2+x)\sum_{n \leq x} |a_n|^2.
\end{eqnarray*}
This completes the proof of \thmref{thmN-abLS}.  \qed

\medspace

\subsection{Proof of \thmref{thmNLSunc}}
Let $A(x)$ be the set of primes $p \leq x$ satisfying 
$i_p \leq x^{\frac{1}{2}} \log x$ 
and $(a_n)_n$ be the indicator function of $A(x)$. 
Also let $f$ be an arithmetic function defined by 
$f(n) = i_n$ if $n \in A(x)$ and $f(n)=1$ otherwise. 
Then we have
$$
1 ~\leq~ f(n) ~\leq~  x^{\frac{1}{2}} \log x
$$
for all $n \leq x$. Now \thmref{thmNLSunc} follows by arguing as 
in the proof of \thmref{thmN-abLS}. \qed

\medspace

In the next subsection, we will apply the duality principle 
to derive some improvements for the non-abelian large sieve inequality.

\medspace

\subsection{A duality principle and  large sieve inequality}
We prove the following variation of the duality principle.
\begin{thm}\label{thmDualVar}
Let $(c_{ij})$ be complex numbers for $1 \leq i \leq n$, $1 \leq j \leq m$ and 
$g: \N \times \N \to \{0, 1\}$ be any function. 
Also let $\Delta$ be a positive real number. 
Then we have
$$
\sum_{1 \leq j \leq m} 
\left| \sum_{1 \leq i \leq n  \atop g(i, j) = 1} c_{ij} a_i \right|^2
~\leq~
\Delta \sum_{1 \leq i \leq n} |a_i|^2
$$
for any sequence $(a_i)_{i=1}^{n}$ of complex numbers if and only if
$$
\sum_{1 \leq i \leq n}
\left| \sum_{1 \leq j \leq m \atop g(i, j) = 1}   c_{ij} b_j \right|^2
~\leq~
\Delta \sum_{1 \leq j \leq m} |b_j|^2
$$
for any sequence $(b_j)_{j=1}^{m}$ of complex numbers.
\end{thm}
\begin{proof}
\thmref{thmDualVar} follows from \thmref{thmDual} by setting
$$
\tilde{c}_{ij} ~=~ 
\begin{cases}
c_{ij} & \text{if } g(i, j) = 1, \\
0 & \text{if } g(i, j) = 0.
\end{cases}
$$
\end{proof}

\begin{rmk}
Before moving further,  we will make the following observation.
The first sum in \eqref{eqphiqd0} is
\begin{eqnarray*}
\sum_{q \leq Q} \frac{1}{q^2} \left| \sideset{}{'}\sum_{p \leq x}
\sum_{c, d \in \F_q^{^*} } 
e\(\frac{(p-1) d - i_p c}{q}\) \right|^2
&~=~&
\sum_{q \leq Q} \frac{1}{q^2} 
\left|
\sideset{}{'}\sum_{p \leq x \atop q \mid i_p} \sum_{c, d \in \F_q^{^*}} 1 
~+~ \sideset{}{'}\sum_{p \leq x  \atop q \nmid i_p} 
\sum_{c, d \in \F_q^{^*}} e\(\frac{(p-1) d - i_p c}{q}\)
\right|^2.
\end{eqnarray*}
From \eqref{eqd0BT}, we get
\begin{eqnarray*}
\sum_{q \leq Q} \frac{1}{q^2} 
\left|
\sideset{}{'}\sum_{p \leq x \atop q \mid i_p} \sum_{c, d \in \F_q^{^*}} 1 
\right| ^2
~\ll~ 
\sum_{q \leq Q} \frac{1}{q^2} 
\left|
\sideset{}{'}\sum_{p \leq x \atop p \equiv 1 (\mod q)} q^2
\right|^2
~\ll~ Q \pi(x)^2.
\end{eqnarray*}
Hence we get
\begin{eqnarray*}\label{eqqnmidip}
\sum_{q \leq Q} \frac{1}{q^2} \left| \sideset{}{'}\sum_{p \leq x}
\sum_{c, d \in \F_q^{^*} } 
e\(\frac{(p-1) d - i_p c}{q}\) \right|^2
&~\ll~&
Q\pi(x)^2 ~+~ \sum_{q \leq Q} \frac{1}{q^2} 
\left|
\sideset{}{'}\sum_{p \leq x  \atop q \nmid i_p} 
\sum_{c, d \in \F_q^{^*}} e\(\frac{(p-1) d - i_p c}{q}\)
\right|^2 \\
&~\ll~&
Q\pi(x)^2 ~+~ \sum_{q \leq Q} \frac{1}{q^2} 
\left|
\sum_{c, d \in \F_q^{^*}}  e \(- \frac{d}{q}\) 
\sideset{}{'}\sum_{p \leq x  \atop q \nmid i_p} e\(\frac{ pd - i_p c}{q}\)
\right|^2.
\end{eqnarray*}
By applying the Cauchy-Schwarz inequality, we get
\begin{eqnarray*}
\sum_{q \leq Q} \frac{1}{q^2} 
\left|
\sum_{c, d \in \F_q^{^*}}  e \(- \frac{d}{q}\) 
\sideset{}{'}\sum_{p \leq x  \atop q \nmid i_p} e\(\frac{ pd - i_p c}{q}\)
\right|^2 
&~\ll~& 
\sum_{q \leq Q} \frac{1}{q^2} \cdot (q-1)^2 
\sum_{c, d \in \F_q^{^*}} 
\left|
\sideset{}{'}\sum_{p \leq x  \atop q \nmid i_p} e\(\frac{ pd - i_p c}{q}\)
\right|^2 \\
&~\ll~& 
\sum_{q \leq Q} \sum_{c, d \in \F_q^{^*}} 
\left|
\sideset{}{'}\sum_{p \leq x  \atop (q, i_p)=1} e\(\frac{ pd + i_p c}{q}\)
\right|^2.
\end{eqnarray*}
Thus we get
\begin{eqnarray}\label{eqphiqnls}
\sum_{q \leq Q} \left| \sideset{}{'}\sum_{p \leq x} \varphi_q(\sigma_p) \right|^2 
&~\ll~&
Q \pi(x)^2 ~+~ \sum_{q \leq Q} \sum_{c, d \in \F_q^{^*}} 
\left|
\sideset{}{'}\sum_{p \leq x  \atop (q, i_p)=1} e\(\frac{ pd + i_p c}{q}\)
\right|^2.
\end{eqnarray}
\end{rmk}

\subsection{Proof of \thmref{thmNLS}}
By using a variation of the duality principle (see \thmref{thmDualVar}), 
we need to estimate 
\begin{eqnarray*}
\sum_{n \leq x} \left|
\sum_{q \leq Q \atop (q, f(n))=1} 
\sum_{c, d \in \F_q^{^*}} A(q,c,d) e\( \frac{nd+f(n)c}{q} \)
\right|^2
&~=~&
\sum_{m \leq Cx^{\theta}} \sum_{n \leq x \atop f(n)=m} \left|
\sum_{q \leq Q \atop (q, m)=1} 
\sum_{c, d \in \F_q^{^*}} A(q,c,d)~e\( \frac{nd+mc}{q} \)
\right|^2 
\notag \\ 
&~\leq~&
\sum_{m \leq Cx^{\theta}} \sum_{n \leq x \atop n \equiv 1 (\mod m)} \left|
\sum_{q \leq Q \atop (q, m)=1} 
\sum_{c, d \in \F_q^{^*}} A(q,c,d)~ e\( \frac{nd+mc}{q} \)
\right|^2.
\end{eqnarray*}
The sum on the right hand side of the above equation is 
\begin{eqnarray}
&~& \sum_{m \leq Cx^{\theta}} \sum_{n \leq x \atop n \equiv 1 (\mod m)} 
\left|
\sum_{q \leq Q \atop (q, m)=1} 
\sum_{c, d \in \F_q^{^*}} A(q,c,d) ~e\( \frac{nd+mc}{q} \)
\right|^2 \notag \\
&~=~&
 \sum_{m \leq Cx^{\theta}} \sum_{n \leq x \atop n \equiv 1 (\mod m)}  
 \sum_{q_1, q_2 \leq Q \atop (m, q_1q_2)=1} 
 \sum_{c_1, d_1 \in \F_{q_1}^{^*} \atop c_2, d_2 \in \F_{q_2}^{^*}} 
 A(q_1, c_1, d_1) \overline{A(q_2, c_2, d_2)} \cdot 
 e\( n\(\frac{d_1}{q_1} - \frac{d_2}{q_2}\) + m\(\frac{c_1}{q_1} - \frac{c_2}{q_2}\)\) \notag \\
 &~=~& 
\sum_{q_1, q_2 \leq Q} \sum_{c_1, d_1 \in \F_{q_1}^{^*} \atop c_2, d_2 \in \F_{q_2}^{^*}} 
A(q_1, c_1, d_1) \overline{A(q_2, c_2, d_2)} 
\sum_{m \leq Cx^{\theta} \atop (m, q_1 q_2) = 1} 
e\( m\(\frac{c_1}{q_1} - \frac{c_2}{q_2}\) \)
\sum_{n \leq x \atop n \equiv 1 (\mod m) } 
e\( n\(\frac{d_1}{q_1} - \frac{d_2}{q_2}\) \) \notag \\
&~=~& 
S_1 ~+~ S_2,
\end{eqnarray}
where 
\begin{eqnarray*}
S_1 ~=~
\sum_{q \leq Q} \sum_{c_1, c_2 \in \F_q^{^*} \atop d \in \F_q^{^*}}  
A(q, c_1, d) \overline{A(q, c_2, d)} 
\sum_{m \leq Cx^{\theta} \atop (m, q) = 1} e\( m\(\frac{c_1 - c_2}{q}\) \)
\sum_{n \leq x \atop n \equiv 1 (\mod m) } 1
\end{eqnarray*}
and
\begin{eqnarray*}
S_2 ~=~
\sum_{q_1, q_2 \leq Q} \sum_{c_1, d_1, c_2, d_2  \atop (q_1, d_1) \neq (q_2, d_2)} 
A(q_1, c_1, d_1) \overline{A(q_2, c_2, d_2)} 
\sum_{m \leq Cx^{\theta} \atop (m, q_1 q_2) = 1} 
e\( m\(\frac{c_1}{q_1} - \frac{c_2}{q_2}\) \)
\sum_{n \leq x \atop n \equiv 1 (\mod m) } 
e\( n\(\frac{d_1}{q_1} - \frac{d_2}{q_2}\) \).
\end{eqnarray*}

\subsection*{Estimation of $S_1$}
We have
\begin{eqnarray*}
S_1 &~=~& 
\sum_{q \leq Q} \sum_{c_1, c_2 \in \F_q^{^*} \atop d \in \F_q^{*}}  
A(q, c_1, d) \overline{A(q, c_2, d)} 
\sum_{m \leq Cx^{\theta} \atop (m, q) = 1} e\( m\(\frac{c_1 - c_2}{q}\) \)
\(\frac{x}{m} ~+~ O(1)\) \\
 &~=~& 
x \sum_{q \leq Q} \sum_{c_1, c_2 \in \F_q^{^*} \atop d \in \F_q^{^*}}  
A(q, c_1, d) \overline{A(q, c_2, d)} 
\sum_{m \leq Cx^{\theta} \atop (m, q) = 1} 
\frac{1}{m} e\( m\(\frac{c_1 - c_2}{q} \) \)
~+~
O\(Q x^{\theta} \sum_{q \leq Q} \sum_{c, d \in \F_q^{^*}}  
|A(q, c, d)|^2\) \\
&~=~&
x \sum_{q \leq Q} \sum_{c_1 \neq c_2 \in \F_q^{^*} \atop d \in \F_q^{^*}}  
A(q, c_1, d) \overline{A(q, c_2, d)} 
\sum_{m \leq Cx^{\theta} \atop (m, q) = 1} 
\frac{1}{m} e\( m\(\frac{c_1 - c_2}{q} \) \)\\
&~&
~+~
O\( \(x \log x ~+~ Q x^{\theta} \)\sum_{q \leq Q} \sum_{c, d \in \F_q^{^*}}  
|A(q, c, d)|^2\). 
\end{eqnarray*}
For $c_1 \neq c_2 \in \F_q^{^*}$, consider the sum 
\begin{eqnarray*}
\sum_{m \leq Cx^{\theta} \atop (m, q) = 1} 
\frac{1}{m} e\( m\(\frac{c_1 - c_2}{q}\) \)
~=~
\sum_{m \leq Cx^{\theta}} \frac{1}{m} e\( m\(\frac{c_1 - c_2}{q}\) \) 
~+~ O\(\frac{\log x}{q}\).
\end{eqnarray*}
Note that
$$
\sum_{m \leq y} e\( m\(\frac{c_1 - c_2}{q}\) \) 
~\ll~
\min\left\{ y,~ \frac{1}{\norm{\frac{c_1 - c_2}{q}}}\right\},
$$
where $\norm{\cdot}$ is defined by 
$\norm{\alpha} = \inf_{u \in \Z} |\alpha - u|$ for any $\alpha \in \R$.
Set 
$$
\delta((c_1 - c_2)/q) ~=~ \frac{1}{\norm{\frac{c_1 - c_2}{q}}}.
$$
By partial summation,
\begin{eqnarray*}
\sum_{m \leq Cx^{\theta}} \frac{1}{m} e\( m\(\frac{c_1 - c_2}{q}\) \)
&~\ll~&
1 ~+~ \int_{1}^{Cx^{\theta}} \min\left\{ t,~ \delta((c_1-c_2)/q)\right\} \frac{dt}{t^2} \\
&~\ll~&
1 ~+~ \int_{1}^{\delta((c_1 - c_2)/q)} \frac{dt}{t} 
~+~ \int_{\delta((c_1-c_2)/q)}^{Cx^{\theta}} \delta((c_1-c_2)/q) \frac{dt}{t^2} \\
&~\ll~&
\log \(\delta((c_1 - c_2)/q) \).
\end{eqnarray*}
Hence, we  get
\begin{eqnarray*}
\sum_{m \leq Cx^{\theta} \atop (m, q) = 1} \frac{1}{m} e\( m\(\frac{c_1 - c_2}{q}\) \)
~\ll~
\log\(\frac{1}{\norm{\frac{c_1 - c_2}{q}}}\) ~+~ \frac{\log x}{q}.
\end{eqnarray*}
Thus, we have
\begin{eqnarray*}
&~&x \sum_{q \leq Q} \sum_{c_1 \neq c_2 \in \F_q^{^*} \atop d \in \F_q^{^*}}  
A(q, c_1, d) \overline{A(q, c_2, d)} 
\sum_{m \leq Cx^{\theta} \atop (m, q) = 1} \frac{1}{m} e\( m\(\frac{c_1 - c_2}{q} \) \) \\
&~\ll~&
x \sum_{q \leq Q} \sum_{c_1 \neq c_2 \in \F_q^{^*} \atop d \in \F_q^{^*}}  
|A(q, c_1, d)| |A(q, c_2, d)|
\(\log\(\frac{1}{\norm{\frac{c_1 - c_2}{q}}}\) ~+~ \frac{\log x}{q}\) \\
&~\ll~&
x \sum_{q \leq Q} \sum_{c_1 \neq c_2 \in \F_q^{^*} \atop d \in \F_q^{^*}}  
|A(q, c_1, d)|^2 \log\(\frac{1}{\norm{\frac{c_1 - c_2}{q}}}\) 
~+~ x \log x \sum_{q \leq Q} 
\sum_{c_1 \neq c_2 \in \F_q^{^*} \atop d \in \F_q^{^*}} \frac{|A(q, c_1, d)|^2}{q} \\
&~\ll~&
x \sum_{q \leq Q} \sum_{c_1 \in \F_q^{^*} \atop d \in \F_q^{^*}}  
|A(q, c_1, d)|^2 \sum_{c_2 \in \F_q^{{^*}} \atop c_2 \neq c_1} 
\log\(\frac{1}{\norm{\frac{c_1 - c_2}{q}}}\) 
~+~
x \log x \sum_{q \leq Q} \sum_{c_1, d \in \F_q^{^*}}  
|A(q, c_1, d)|^2.
\end{eqnarray*}
Note that
$$
\sum_{c_2 \in \F_q^{{^*}} \atop c_2 \neq c_1} 
\log\(\frac{1}{\norm{\frac{c_1 - c_2}{q}}}\) 
~\ll~
\sum_{n \leq q} \log \(\frac{1}{n/q}\) ~\ll~ q.
$$
Thus, we get
$$
x \sum_{q \leq Q} \sum_{c_1 \neq c_2 \in \F_q^{^*} \atop d \in \F_q^{^*}}  
A(q, c_1, d) \overline{A(q, c_2, d)} 
\sum_{m \leq Cx^{\theta} \atop (m, q) = 1} \frac{1}{m} e\( m\(\frac{c_1 - c_2}{q} \) \) 
~\ll~ Qx \sum_{q \leq Q} \sum_{c, d \in \F_q^{^*}}  
|A(q, c, d)|^2.
$$
Hence, we get
$$
S_1 ~\ll~ Qx \sum_{q \leq Q} \sum_{c, d \in \F_q^{^*}}  |A(q, c, d)|^2.
$$

\subsection*{Estimation of $S_2$}
Suppose that $(q_1, d_1) \neq (q_2, d_2)$. If $(m, q_1q_2) = 1$, then we have
\begin{eqnarray*}
\sum_{n \leq x \atop n \equiv 1 (\mod m) } e\( n\(\frac{d_1}{q_1} - \frac{d_2}{q_2}\) \)
&~=~&
\sum_{0 \leq t \leq \frac{x-1}{m}} e\( (1+mt)\(\frac{d_1}{q_1} - \frac{d_2}{q_2}\) \) \\
&~=~&
e\( \frac{d_1}{q_1} - \frac{d_2}{q_2} \) 
\sum_{0 \leq t \leq \frac{x-1}{m}} e\( m \(\frac{d_1}{q_1} - \frac{d_2}{q_2}\) t \).
\end{eqnarray*}
Hence, we get
\begin{eqnarray*}
\left|
\sum_{n \leq x \atop n \equiv 1 (\mod m) } e\( n\(\frac{d_1}{q_1} - \frac{d_2}{q_2}\) \)
\right| 
~\ll~
\frac{1}{\left| e\( m \(\frac{d_1}{q_1} - \frac{d_2}{q_2}\) \) - 1 \right|} 
~\ll~
\frac{1}{\left| \sin\(\pi  m \(\frac{d_1}{q_1} - \frac{d_2}{q_2}\)\) \right|} 
~\ll~
\frac{1}{\norm{m \(\frac{d_1}{q_1} - \frac{d_2}{q_2}\)}}.
\end{eqnarray*}
Hence, we get
\begin{eqnarray*}
S_2 &~\ll~&
\sum_{q_1, q_2 \leq Q \atop c_1, d_1 \in \F_{q_1}^{^*}} 
\sum_{c_2, d_2 \in \F_{q_2}^{^*} \atop (q_1, d_1) \neq (q_2, d_2)} 
|A(q_1, c_1, d_1) A(q_2, c_2, d_2)|
\sum_{m \leq Cx^{\theta} \atop (m, q_1 q_2) = 1} 
\frac{1}{\norm{m \(\frac{d_1}{q_1} - \frac{d_2}{q_2}\)}} \\
&~\ll~&
\sum_{m \leq Cx^{\theta} } 
\sum_{q_1, q_2 \leq Q \atop (m, q_1 q_2) = 1} 
\sum_{c_1, d_1 \in \F_{q_1}^{^*} }
\sum_{c_2, d_2 \in \F_{q_2}^{^*} \atop (q_1, d_1) \neq (q_2, d_2)} 
\frac{|A(q_1, c_1, d_1)|^2 ~+~ |A(q_2, c_2, d_2)|^2}{\norm{m \(\frac{d_1}{q_1} - \frac{d_2}{q_2}\)}} \\
&~\ll~&
\sum_{m \leq Cx^{\theta}} 
\sum_{q_1 \leq Q \atop (m, q_1) = 1} 
\sum_{c_1, d_1 \in \F_{q_1}^{^*} } |A(q_1, c_1, d_1)|^2
\sum_{q_2 \leq Q \atop (m, q_2) = 1} 
\sum_{c_2, d_2 \in \F_{q_2}^{^*} \atop (q_2, d_2) \neq (q_1, d_1)} 
\frac{1}{\norm{m \(\frac{d_1}{q_1} - \frac{d_2}{q_2}\)}}.
\end{eqnarray*}
The inner sum is
\begin{eqnarray*}
\sum_{q_2 \leq Q \atop (m, q_2) = 1} 
\sum_{c_2, d_2 \in \F_{q_2}^{^*} \atop (q_2, d_2) \neq (q_1, d_1)} 
\frac{1}{\norm{m \(\frac{d_1}{q_1} - \frac{d_2}{q_2}\)}} 
~\ll~
Q \sum_{q_2 \leq Q \atop (m, q_2) = 1} 
\sum_{d_2 \in \F_{q_2}^{^*} \atop (q_2, d_2) \neq (q_1, d_1)} 
\frac{1}{\norm{\frac{md_1}{q_1} - \frac{md_2}{q_2}}} 
~\ll~
Q \sum_{t \leq Q^2} \frac{1}{t/Q^2} 
~\ll~ Q^3 \log Q.
\end{eqnarray*}
Thus, we get
$$
S_2 ~\ll~ Q^3 x^{\theta} \log Q \sum_{q_1 \leq Q} 
\sum_{c_1, d_1 \in \F_{q_1}^{^*}}  |A(q_1, c_1, d_1)|^2.
$$
Hence, we get
\begin{eqnarray*}
\sum_{n \leq x} \left|
\sum_{q \leq Q \atop (q, f(n))=1} \sum_{c, d \in \F_q^{{^*}}} A(q,c,d) 
e\( \frac{nd+f(n)c}{q} \)
\right|^2 
&~\ll~& 
\(Q x ~+~ Q^3 x^{\theta} \log Q\) \sum_{q \leq Q} \sum_{c, d \in \F_q^{^*}}  
|A(q, c, d)|^2.
\end{eqnarray*}
Thus, by duality principle (see \thmref{thmDualVar}), we deduce that
\begin{equation*}
\sum_{q \leq Q} \sum_{c, d \in \F_q^{{^*}}} \left|
\sum_{n \leq x \atop (q, f(n))=1} a_n e\( \frac{nd+f(n)c}{q} \)
\right|^2
~\ll~(Qx ~+~ Q^3 x^\theta \log Q) \sum_{n \leq x} |a_n|^2. \qed    
\end{equation*}

\medspace

We now exploit the fact that $\varphi_q(\sigma_p)$ is supported on 
primes $p \equiv 1~(\mod q)$ to give further estimates for non-abelian large sieve inequality.
From \eqref{eqphisigmap}, we have $\varphi_q(\sigma_p) \neq 0$ 
implies that $p \equiv 1 (\mod q)$. Further, if $p \equiv 1 (\mod q)$, then we have
\begin{eqnarray*}
\varphi_q(\sigma_p) ~=~ \sum_{c \in \F_{q}^{*}} e\(\frac{i_p c}{q}\).
\end{eqnarray*}
In this context, we have  the following lemma.
\begin{lem}\label{EqPhiq=1(q)}
We have
$$
\sum_{q \leq Q} \left| \sum_{p \leq x} \varphi_q(\sigma_p) \right|^2
~\leq~ Q \sum_{q \leq Q} \sum_{c \in \F_{q}^{^*}} \left|
\sum_{p \leq x \atop p \equiv 1 (\mod q)} e\(\frac{i_p c}{q}\)
\right|^2.
$$
\end{lem}
\begin{proof}
We have
\begin{eqnarray*}
\sum_{q \leq Q} \left| \sum_{p \leq x} \varphi_q(\sigma_p) \right|^2
&~=~&
\sum_{q \leq Q} \left| \sum_{p \leq x \atop p \equiv 1 (\mod q)} 
\sum_{c \in \F_{q}^{^*}} e\(\frac{i_p c}{q}\) \right|^2 
~=~
\sum_{q \leq Q} \left| \sum_{c \in \F_{q}^{^*}} 
\sum_{p \leq x \atop p \equiv 1 (\mod q)}  e\(\frac{i_p c}{q}\) \right|^2 \\
&~\leq~&
\sum_{q \leq Q}  \varphi(q) \sum_{c \in \F_{q}^{^*}} 
\left| \sum_{p \leq x \atop p \equiv 1 (\mod q)}  e\(\frac{i_p c}{q}\) \right|^2,
\end{eqnarray*}
completing the proof of \lemref{EqPhiq=1(q)}.
\end{proof}

\medspace

\subsection{Proof of \thmref{thm-NLS-opt}}
By Cauchy-Schwarz inequality, we get
\begin{eqnarray*}
\sum_{q \leq Q} \sum_{c \in \F_{q}^{^*}} \left|
\sum_{p \leq x \atop p \equiv 1 (\mod q)} a_p~e\(\frac{i_p c}{q}\)
\right|^2
~\ll~
\sum_{q \leq Q} \sum_{c \in \F_{q}^{^*}} 
\(\sum_{p \leq x \atop p \equiv 1 (\mod q)} |a_p|^2\)
\(\sum_{p \leq x \atop p \equiv 1 (\mod q)} 1\).
\end{eqnarray*}
By applying Brun-Titchmarsh inequality, we get
\begin{eqnarray*}
\sum_{q \leq Q} \sum_{c \in \F_{q}^{^*}} \left|
\sum_{p \leq x \atop p \equiv 1 (\mod q)} a_p~e\(\frac{i_p c}{q}\)
\right|^2
&~\ll_{\epsilon}~&
\sum_{q \leq Q} \sum_{c \in \F_{q}^{^*}} \frac{\pi(x)}{q} 
\sum_{p \leq x \atop p \equiv 1 (\mod q)} |a_p|^2 
~\ll_{\epsilon}~
\pi(x) \sum_{q \leq Q} \sum_{p \leq x \atop p \equiv 1 (\mod q)} |a_p|^2 \\
&~\ll_{\epsilon}~& 
\pi(x) \sum_{p \leq x} |a_p|^2 \sum_{q \leq Q \atop q \mid p-1} 1 
~\ll_{\epsilon}~
\frac{x}{\log\log x} \sum_{p \leq x} |a_p|^2,
\end{eqnarray*}
since
$$
\sum_{q \leq Q \atop q \mid p-1} 1  ~\leq~ 
\sum_{q \leq x \atop q \mid p-1} 1  ~\ll~ \frac{\log x}{\log\log x}.
$$
Now suppose that the Generalized Riemann Hypothesis (GRH) is true. Set
$$
b_p ~=~
\begin{cases}
1 & \text{ if $2$ is a primitive root }(\mod p), \\
0 & \text{otherwise.}
\end{cases}
$$
Then we have
\begin{eqnarray*}
\sum_{q \leq Q} \sum_{c \in \F_{q}^{^*}} \left|
\sum_{p \leq x \atop p \equiv 1 (\mod q)} b_p~e\(\frac{i_p c}{q}\)
\right|^2
~=~
\sum_{q \leq Q} \varphi(q) 
\(\sum_{ \substack{p \leq x \\ i_p =1 \\ p \equiv 1 (\mod q)}} 1\)^2.
\end{eqnarray*}
If $\log Q= o(\log x)$, then we have
(see \cite[Theorem 1.2]{Mo} and \cite[Theorem 1.1]{Zo})
$$
\sum_{ \substack{p \leq x \\ i_p =1 \\ p \equiv 1 (\mod q)}} 1 
~\gg~ \frac{\pi(x)}{\varphi(q)}.
$$
Hence we deduce that
$$
\sum_{q \leq Q} \sum_{c \in \F_{q}^{^*}} \left|
\sum_{p \leq x \atop p \equiv 1 (\mod q)} b_p~e\(\frac{i_p c}{q}\)
\right|^2
~\gg~ \pi(x)^2 \log\log Q 
~\gg~ \pi(x) \log\log Q \sum_{p \leq x} |b_p|^2.
$$
This completes the proof of \thmref{thm-NLS-opt}. \qed

\section*{Acknowledgments}
The project was initiated when the second author was a Coleman postdoctoral fellow at Queen's University, Kingston, Canada. We acknowledge Queen’s University, Canada, for providing an excellent atmosphere to work. The second author thanks the Max Planck Institute for Mathematics, Bonn, Germany, for its kind hospitality. We thank Pieter Moree for his comments on an earlier version of the manuscript.

\end{document}